\documentclass[reqno]{amsart}
\usepackage{graphicx} 
\usepackage{amsmath}%
\usepackage{amsfonts}%
\usepackage{amssymb}%
\usepackage{graphicx}
\usepackage{mathrsfs}
\usepackage{hyperref}
\usepackage{fullpage}
\usepackage{cite}
\usepackage{mathrsfs}
\usepackage{amsmath}
\usepackage{amssymb}
\usepackage{amsthm}
\usepackage{amsfonts}
\usepackage{amssymb}
\usepackage{amsmath}
\usepackage{enumitem}
\usepackage{mathtools}
\usepackage{mathrsfs}
\usepackage{color}
\usepackage{graphicx}

\allowdisplaybreaks
\def\0{\emptyset}

\def\0{\emptyset}

\newtheorem{theorem}{Theorem}
\newtheorem{lemma}[theorem]{Lemma}

\newtheorem{claim}[theorem]{Claim}

\newtheorem{conjecture}[theorem]{Conjecture}
\newtheorem{definition}[theorem]{Definition}

\newtheorem{problem}[theorem]{Problem}
\newtheorem{remark}[theorem]{Remark}

\usepackage{graphicx}
\usepackage{amsmath}
\usepackage{color}

\begin{document}
\begin{center}
{\LARGE Shadows of Uniform Hypergraphs under a Minimum Degree Condition}%
\end{center}

\vspace{0.3em}

\begin{center}
{\large Haorui Liu\,\textsuperscript{a}, \ Mei Lu\,\textsuperscript{a,\,1}, \ Yi Zhang\,\textsuperscript{b,c,\,$\ast$,\,2}}
\end{center}

\vspace{0.6em}

\begin{center}
\begin{minipage}{0.92\textwidth}
\small
\begin{tabular}{@{}r@{\ }p{0.95\linewidth}@{}}
\textsuperscript{a} & Department of Mathematical Sciences, Tsinghua University, Beijing, 100084, P.R.\ China \\[0.25em]
\textsuperscript{b} & School of Mathematical Sciences, Beijing University of Posts and Telecommunications, Beijing, 100876, P.R.\ China \\[0.25em]
\textsuperscript{c} & Key Laboratory of Mathematics and Information Networks (BUPT), Ministry of Education, Beijing, 100876, P.R.\ China
\end{tabular}
\end{minipage}
\end{center}

\renewcommand{\thefootnote}{\fnsymbol{footnote}}
\setcounter{footnote}{1}
\footnotetext[1]{Corresponding author: Yi Zhang.\\
\emph{E-mail addresses:} \texttt{lhr22@mails.tsinghua.edu.cn} (H.\ Liu), \texttt{lumei@tsinghua.edu.cn} (M.\ Lu), \texttt{shouwangmm@sina.com} (Y.\ Zhang).}
\renewcommand{\thefootnote}{\arabic{footnote}}
\setcounter{footnote}{0}
\footnotetext[1]{Mei Lu was supported by National Natural Science Foundation of China under Grant No.~12571372 and Beijing Natural Science Foundation 1262010.}
\footnotetext[2]{Yi Zhang is supported by Fundamental Research Funds for the Central Universities and Innovation Foundation of BUPT for Youth under Grant No.~2023RC49 and National Natural Science Foundation of China under Grant No.~11901048 and 12071002.}

\vspace{1.2em}

\begin{center}
\begin{minipage}{0.92\textwidth}
\small
\noindent\textbf{Abstract.} Given a set $X$ and an integer $t$, let $\mathcal{F}$ be a family of $k$-subsets of $X$. The Kruskal--Katona theorem implies that if $|\mathcal{F}|\geq \binom{t}{k}$, then $|\partial_\ell\mathcal{F}|\geq\binom{t}{\ell}$. The minimum degree version of this problem asks: if $\delta(\mathcal{F})\geq \binom{t}{k-1}$, how small can $|\partial_\ell\mathcal{F}|$ be? We call a hypergraph \textit{extremal} if it achieves the minimum value of $|\partial_\ell \mathcal{F}|$ subject to the degree condition $\delta(\mathcal{F}) \geq \binom{t}{k-1}$. F\"uredi and Zhao [SIAM J. Discrete Math. 36(4), 2022] proved that for $k=3$, $\ell=2$ and $t\ge 2$, every extremal hypergraph contains an isolated copy of $K_{t+1}^3$ when $|X| > \frac{1}{4}(t+1)^2(t+2)$. In this article, we study the general case $k > \ell \geq 2$. By developing a hypergraph transformation that combines shifting operations with antilexicographic compression, we prove that, for every integer \(t\ge k-1\),  there exists an extremal hypergraph containing an isolated copy of $K^{k}_{t+1}$ whenever $|X| > \frac{1}{4}(t+1)^2\binom{t-1}{\ell-2} + 3t+1$. In the case when \(k=3\) and \(\ell=2\), this gives the threshold
\(\frac14(t+1)^2+3t+1\), which is smaller than
\(\frac14(t+1)^2(t+2)\) for every \(t\ge3\); for \(t=2\), the two thresholds
give the same integer condition on \(|X|\).
\end{minipage}
\end{center}

\vspace{0.8em}

\begin{center}
\begin{minipage}{0.92\textwidth}

\end{minipage}
\end{center}

\medskip

\section{Introduction}

Let $X$ be a finite set and let $\binom{X}{k}$ denote the collection of all $k$-subsets of $X$. 
For a family $\mathcal{F} \subseteq \binom{X}{k}$, the $\ell$-shadow of $\mathcal{F}$, denoted by $\partial_\ell \mathcal{F}$, 
is defined as the family of all $\ell$-subsets of $X$ that are contained in at least one member of $\mathcal{F}$.

\begin{definition}
For any labeled set $X = [n]$ and any integer $k \le n$, the \textit{antilexicographic order} ``$<$" on $\binom{X}{k}$ is defined by setting, for $A, B \in \binom{X}{k}$,  $A < B$ if $\max\{i : i \in A \setminus B\} < \max\{i : i \in B \setminus A\}$; we write $A \le B$ if $A < B$ or $A = B$. For any real number $x \ge k$, we define the generalized binomial coefficient as $\binom{x}{k} = \frac{x(x-1)\cdots(x-k+1)}{k!}$.
\end{definition}

\begin{theorem}[Kruskal-Katona] The following two versions of the theorem will be employed in our analysis:

\begin{enumerate}
    \item[(i)] \textbf{(Original Version \cite{Kr,Ka})} Suppose $\mathcal{F} \subseteq \binom{X}{k}$ and $|\mathcal{F}| = m$. Then $|\partial_\ell \mathcal{F}| \ge |\partial_\ell \mathcal{F}_0|$, where $\mathcal{F}_0$ is the collection of the $m$ smallest elements in $\binom{X}{k}$ with respect to the antilexicographic order.
    
    \item[(ii)] \textbf{(Lovász Version \cite{Lovász})} If $\mathcal{F}$ is a $k$-uniform family with $|\mathcal{F}| \ge \binom{x}{k}$ for some real $x \ge k$, then for any $\ell \le k$, the size of its $\ell$-shadow satisfies 
    \[ |\partial_\ell \mathcal{F}| \ge \binom{x}{\ell}, \]
    and equality occurs if and only if $x$ is an integer and $\mathcal{F}$ is the family of all $k$-subsets of some $x$-element set.
\end{enumerate}
\end{theorem}

A family $\mathcal{F}$ of $k$-subsets of $X$ can also be regarded as a $k$-uniform hypergraph with vertex set $X$. A $k$-uniform hypergraph (briefly, $k$-graph) $H$ is a pair $(V, E)$, where $V = V (H)$ is a finite set of vertices and $E = E(H)$ is a family of $k$-element subsets of $V$. Given a set $S \subseteq V$, the degree $\deg_H(S)$ of $S$ is the number of edges of $H$ that contain $S$. We write $\deg(S)$ when $H$ is clear from the context. If $S = \{u\}$, then we write $\deg(u)$ instead of $\deg(\{u\})$. Let $\delta(H) = \min\{ \deg(u) : u \in V (H)\}$.  For any $Y, S \subseteq V$ with $0 \le |Y| < k$, the \textit{link} of $Y$ restricted to $S$, denoted by $L_H^Y(S)$, is defined as the $(k-|Y|)$-graph on the vertex set $S$ with edge set 
$E(L_H^Y(S)) = \{ A \subseteq S : A \cup Y \in E(H) \}.$ 
For brevity, we denote $L_H^{\{u\}}(V(H))$ simply by $L_H(u)$,   and when there is no confusion, we use $L_H^Y(S)$ to denote its edge set $E(L_H^Y(S))$. The \textit{neighborhood} of a vertex $u$ is defined as 
$N_H(u) = \{v \in V(H) : v \neq u \text{ and } \exists\, e \in E(H) \text{ with } \{u,v\} \subseteq e\}$ and for any $v\in V(H)$, we say $v$ is a \textit{neighbor} 
of $u$ if $v\in N_H(u)$.  More generally, 
for a set $S \subseteq V(H)$, 
the \emph{neighborhood} of $S$ in $H$ is defined as
$N_H(S) = \bigl\{ v \in V(H) \setminus S \; :\; 
    \exists\, e \in E(H) \text{ such that } v \in e 
    \text{ and } e \cap S \neq \emptyset \bigr\}.$
 The \textit{induced subgraph} of $H$ on $S$, 
denoted by $H[S]$, is the $k$-graph with vertex set $S$ and edge set 
$E(H[S]) = \{e \in E(H) : e \subseteq S\}$.
A \textit{clique} of $H$ is a set $S \subseteq V(H)$ such that every 
$k$-element subset of $S$ is an edge of $H$, i.e., $\binom{S}{k} \subseteq E(H)$.

Recently, 
 extremal problems under a minimum or maximum degree condition have received considerable attention. Frankl\cite{Fra2} studied the Erd{\H{o}}s--Ko--Rado theorem under the condition of a maximum degree. Frankl and Tokushige\cite{Fra4}, Huang and Zhao\cite{Huang1}, Kupavskii\cite{Kup} and Huang and Zhang \cite{Huang2} studied the minimum degree version of the Erd{\H{o}}s--Ko--Rado theorem. Frankl, Han, Huang and Zhao\cite{Fra3} studied the minimum degree version of the Hilton-Milner theorem. In addition, there has been extensive research on the degree version of the Erd{\H{o}}s matching conjecture\cite{Han, Kha1, Mar, Kuhn1, Pik, Rod2, Rod3, TrZh12, TrZh13,wang, Yi4, zhang, zhang2} and Tur\'{a}n problems\cite{Lo, Lo1, Si, Mu}. Chase \cite{Ch} studied the Kruskal-Katona theorem under a condition of maximum degree.

 The minimum degree version of the Kruskal--Katona theorem was initiated by F\"uredi and Zhao \cite{Fu2}, who posed the following problem.

\begin{problem}\cite{Fu2}\label{prob:FZ-shadow}
Given integers $k\geq 3$ and \(t\ge k-1\),  let $X$ be a set of $n$ vertices and $\mathcal{F}$ be a family of $k$-subsets of $X$ with $\delta(\mathcal{F}) \geq \binom{t}{k-1}$. Determine the minimum possible size of the $(k-1)$-shadow $|\partial_{k-1} \mathcal{F}|$.
\end{problem}

F\"uredi and Zhao \cite{Fu2} provided the following equivalent formulation of Problem~\ref{prob:FZ-shadow}.

\begin{problem}\cite{Fu2}\label{prob:FZ-hyp}
Given integers \(k\ge3\) and \(t\ge k-1\), determine the minimum number of edges in a $(k-1)$-graph $G$ on $n$ vertices such that every vertex is contained in at least $\binom{t}{k-1}$ copies of $K^{k-1}_k$ (the complete $(k-1)$-graph on  $k$ vertices).
\end{problem}

In the case $k = 3$, Problem~\ref{prob:FZ-hyp} asks for the minimum number of edges in a graph $G$ on $n$ vertices such that every vertex is contained in at least $\binom{t}{2}$ triangles. This is closely related to the classical Rademacher--Tur\'{a}n problem, initiated by Rademacher (unpublished) and Erd\H{o}s \cite{Erdos}, which seeks the minimum number of triangles in a graph of given order and size. A natural variant is to consider the minimum number of triangles containing a fixed vertex.
F\"uredi and Zhao \cite{Fu2} studied this problem and determined the extremal graphs for $t+1 \leq n \leq 2(t+1)$. As the precise statement is rather elaborate, we refer the interested reader to \cite{Fu2} for details. For large $n$, they further proved the following structural result.

\begin{theorem}\cite{Fu2}\label{pro}
    When $n>\frac{1}{4}(t+1)^2(t+2)$ and $t \ge 2$ is an integer, every extremal graph for Problem \ref{prob:FZ-hyp} with $k=3$  contains an isolated copy of $K_{t+1}$.
\end{theorem}

Theorem \ref{pro} implies that in every extremal graph for Problem~\ref{prob:FZ-hyp} with $k=3$, all but at most $\frac14 (t+1)^2 (t+2)$ vertices are contained  in vertex-disjoint copies of $K_{t+1}$. On the other hand, F\"{u}redi and Zhao \cite{Fu2}  determine the extremal graphs for $n \leq 2(t+1)$. There remains a gap between $2(t+1)$ and $\frac{1}{4}(t+1)^2(t+2)$. Thus, closing the gap would completely solve Problem~\ref{prob:FZ-hyp} for $k=3$.  Liu, Lu and Zhang \cite{Liu} narrowed this gap by improving the threshold on $n$ in Theorem~\ref{pro} from $O(t^3)$ to $O(t^2)$ for the stronger conclusion that every extremal graph contains an isolated copy of $K_{t+1}$. In the present paper, we pursue a related but different direction. In the special case $k=3$ and $\ell=2$, our result gives an existence version with a threshold of the same order $O(t^2)$; on the other hand, our method applies to the general setting $k>\ell\ge2$. Specifically, we generalize Problem~\ref{prob:FZ-shadow} to arbitrary $\ell$-shadows and consider the following problem.

\begin{problem}\label{dk}
    Given integers $k> \ell \geq 2$ and \(t\ge k-1\), let $X$ be a set of $n$ vertices and $\mathcal{F}$ be a family of $k$-subsets of $X$ with $\delta(\mathcal{F)}\geq\binom{t}{k-1}$. Determine the minimum possible size of the $\ell$-shadow $|\partial_{\ell}\mathcal{F}|$.
\end{problem}

Equivalently, we may formulate this in terms of hypergraphs.

\begin{problem}\label{13}
    Given integers \(k>\ell\ge2\) and \(t\ge k-1\), determine the minimum number of edges in an $\ell$-graph $G$ on $n$ vertices such that every vertex is contained in at least $\binom{t}{k-1}$ copies of $K_k^{\ell}$.
\end{problem}

 Note that Problems~\ref{prob:FZ-shadow} and~\ref{prob:FZ-hyp} are the special cases $\ell = k-1$ of Problems~\ref{dk} and~\ref{13}, respectively. 
In what follows, we work with Problem \ref{13} as our primary formulation.
Since every vertex $v \in V(G)$ is contained in at least $\binom{t}{k-1}$ copies of $K_k^\ell$, the link graph $L_G(v)$ of each vertex $v$, which is an $(\ell-1)$-graph,  must contain at least $\binom{t}{k-1}$ copies of $K_{k-1}^{\ell-1}$. Applying the Lov\'{a}sz version of the Kruskal-Katona theorem (Theorem 2 (ii)) to the link graph yields $\text{deg}(v) \ge \binom{t}{\ell-1}$, which  immediately gives the lower bound:
\begin{equation}\label{eqation111}
|E(G)| \ge \frac{n}{\ell} \binom{t}{\ell-1}.
\end{equation}
Note that the lower bound in (1) is sharp whenever $t$ is an integer and $(t+1)$ divides $n$, as the equality is attained by a disjoint union of copies of $K_{t+1}^\ell$. However, when $t+1$ does not divide $n$, the vertices cannot be evenly partitioned into copies of $K_{t+1}^{\ell}$, so the lower bound in (\ref{eqation111}) is no longer attainable, and determining the exact minimum becomes significantly more challenging.

The following lemma establishes an upper bound on the total ``excess'' degree sum in any extremal $\ell$-graph, which constrains how far the extremal structure can deviate from a disjoint union of copies of $K^\ell_{t+1}$. In particular, it implies that all but at most $O(t^\ell)$ vertices have degree exactly $\binom{t}{\ell-1}$ and thus each such vertex belongs to a unique copy of $K^\ell_{t+1}$.  This estimate will be used in the proof of the main theorem.

\begin{lemma}\label{d}
Let $G$ be an extremal $\ell$-graph for Problem \ref{13}. We have
\begin{equation*}
\sum_{v \in V(G)} \left( \text{deg}(v) - \binom{t}{\ell-1} \right) \le \frac{1}{4}(t+1)^2 \binom{t-1}{\ell-2}.
\end{equation*}
\end{lemma}

\begin{proof}
Let $n = q(t+1) + r$ with $0 \le r < t+1$ and let $G'$ be a specific $\ell$-graph consisting of $q-1$ disjoint copies of $K_{t+1}^\ell$ and a remaining component formed by two copies of $K_{t+1}^\ell$ sharing $t+1-r$ vertices. Then $G'$ satisfies the condition of Problem \ref{13}. Since $G$ is an extremal $\ell$-graph that minimizes the number of edges for Problem \ref{13}, we must have $|E(G)| \le |E(G')|$. 

First, we determine  $|E(G')|$. By considering the edges within each component, we have:
\begin{equation*}
\ell |E(G')| = n \binom{t}{\ell-1} + (t+1-r) \left[ \binom{t}{\ell-1} - \binom{t-r}{\ell-1} \right].
\end{equation*}
By the handshaking lemma for $\ell$-graphs, $\sum_{v \in V(G)} \text{deg}(v) = \ell |E(G)|$. Therefore, the total ``excess'' degree sum of $G$ can be bounded as follows:
\begin{align*}
\sum_{v \in V(G)} \left( \text{deg}(v) - \binom{t}{\ell-1} \right) &= \ell |E(G)| - n \binom{t}{\ell-1} \\
&\le \ell |E(G')| - n \binom{t}{\ell-1} \\
&= (t+1-r) \left[ \binom{t}{\ell-1} - \binom{t-r}{\ell-1} \right].
\end{align*}
Using the combinatorial identity $\binom{N}{k} - \binom{N-r}{k} = \sum_{i=1}^r \binom{N-i}{k-1}$, we can bound the term in the brackets:
\begin{align*}
(t+1-r) \sum_{i=1}^r \binom{t-i}{\ell-2} &\le r(t+1-r) \binom{t-1}{\ell-2} \le \frac{1}{4}(t+1)^2 \binom{t-1}{\ell-2},
\end{align*}
where the last inequality follows from the fact that the quadratic $r(t+1-r)$ is maximized at $r = \frac{t+1}{2}$.
\end{proof}

We are now ready to state our main result. It gives an existence version of the isolated-clique phenomenon in the general setting $k>\ell\ge2$. In the special case $k=3$ and $\ell=2$, the threshold is of order $O(t^2)$.

\begin{theorem}\label{1}
If $n > \frac{1}{4}(t+1)^2 \binom{t-1}{\ell - 2} + 3t+1$ and $t\geq k-1$ is an integer, then there exists an extremal hypergraph for Problem~\ref{13} that contains an isolated copy of $K^\ell_{t+1}$.
\end{theorem}

 Since Problems~\ref{dk} and~\ref{13} are equivalent, Theorem~\ref{1} is equivalent to the
  following result stated in terms of $k$-uniform hypergraphs.

  \begin{theorem}\label{thm:main-prime}
  If $n > \frac{1}{4}(t+1)^2\binom{t-1}{\ell-2} + 3t+1$ and $t \geq k-1$ is an integer,
  then there exists an extremal $k$-uniform hypergraph for Problem~\ref{dk} that contains an
  isolated copy of $K^k_{t+1}$.
\end{theorem}

\section{A useful construction}


We now introduce a graph transformation, denoted by $\mathcal{G}$, which regularizes the structure of an extremal \(\ell\)-graph while preserving edge minimality and the minimum clique-degree condition.  Let $G$ be an extremal $\ell$-graph for Problem \ref{13} with $|V(G)|>\frac{1}{4}(t+1)^2 \binom{t-1}{\ell-2} + 2t+1$, and let $A_1, A_2 \in \binom{V(G)}{t+1}$ be two vertex sets that both induce a $K_{t+1}^\ell$ in $G$. Setting $A_1 = \{v_1, v_2, \dots, v_{t+1}\}$, we aim to construct a new $\ell$-graph $G'' = \mathcal{G}(G, A_1, A_2)$ with the following properties:

\begin{enumerate}
    \item[(1)] \textbf{Vertex Set Preservation:} $V(G'') = V(G)$.
    \item[(2)] \textbf{Edge Non-increase:} $|E(G'')| \le |E(G)|$.
    \item[(3)] \textbf{Edge Confinement:} Both $A_1$ and $A_2$ induce a $K_{t+1}^\ell$ in $G''$. Furthermore, every edge $e \in E(G'')$ that intersects $A_2 \setminus A_1$ must be a subset of $A_2$ (i.e., if $e \cap (A_2 \setminus A_1) \neq \emptyset$, then $e \subseteq A_2$).
    \item[(4)] \textbf{Clique Shifting:} The construction \textbf{ensures} that for any $A_3 \in \binom{V(G)}{t+1}$ that induces a $K_{t+1}^\ell$ in $G$, its ``shifted'' counterpart 
    \[ A_4 = (A_3 \setminus (A_1 \cup A_2)) \cup \{v_{1}, v_2, \dots, v_{m}\} \]
    also induces a $K_{t+1}^\ell$ in $G''$, where $m = |A_3 \cap (A_1 \cup A_2)|$.
    \item[(5)] \textbf{Maintenance of Clique Condition:} Every vertex in $G''$ remains contained in at least $\binom{t}{k-1}$ copies of $K_k^\ell$.
\end{enumerate}

For convenience, let $V=V(G)$. First, we define an auxiliary $\ell$-graph $G'$ on an expanded vertex set $V \cup C$, designed to satisfy Properties (2)--(5) of $G'' = \mathcal{G}(G,A_1, A_2)$. Subsequently, we apply a sequence of shifting operations to $G'$ to obtain the final graph $G''$ on the original vertex set $V$, thereby satisfying Property (1).
 
Let $A_1$ and $A_2$ both induce a $K_{t+1}^\ell$ in $G$, where $A_1 = \{v_1, \dots, v_{t+1}\}$. 
Let $B$ be the set of vertices in $V \setminus (A_1 \cup A_2)$ having at least one neighbor in $A_1 \cup A_2$. First, we define the auxiliary $\ell$-graph $G'$ as follows:

\begin{enumerate}
    \item[(i)] Let $V(G') = V \cup C$, where $C = \{v_{t+2}, \dots, v_{2(t+1)}\}$, $C \cap V =\emptyset$ and $G'[C]$ is empty.
    \item[(ii)]Induced subgraphs  $G'[A_1]$ and $G'[A_2]$ remain $K_{t+1}^{\ell}$, while all edges $e \subset A_1 \cup A_2$   satisfying both  $e \cap (A_1 \setminus A_2) \neq \emptyset$ and  $e\cap (A_2 \setminus A_1) \neq \emptyset$  are removed.
    \item[(iii)] The induced subgraph on $V \setminus (A_1 \cup A_2)$
 remains unchanged; specifically, $G'[V \setminus (A_1 \cup A_2)] = G[V \setminus (A_1 \cup A_2)]$.
    \item[(iv)]  For any $Y \subset B$ with $0 < |Y| \le \ell - 1$, we have $|L_{G'}^Y(A_1 \cup C)| = |L_G^Y(A_1 \cup A_2)|$ and $L_{G'}^Y(A_2 \setminus A_1) = \emptyset$, ensuring Property (3). Furthermore, the edges in the link graph $L_{G'}^Y(A_1 \cup C)$ are the smallest edges with respect to the antilexicographic order under the labeling $A_1 \cup C = \{v_1, \dots, v_{2(t+1)}\}$.
\end{enumerate}

\begin{remark}
By the construction of $G'$, we have $|E(G')| \le |E(G)|$, satisfying Property (2). One can also verify  that  $G'$ satisfies condition (4). For any $A_3 \in \binom{V}{t+1}$, let $m = |A_3 \cap (A_1 \cup A_2)| > 0$. We show that $A_4 = (A_3 \setminus (A_1 \cup A_2)) \cup \{v_1, \dots, v_m\}$ is a clique in $G'$. Suppose, for the sake of contradiction, that there exists an $\ell$-set $e \subseteq A_4$ such that $e \notin E(G')$. Let $Y = e \cap (A_3 \setminus (A_1 \cup A_2))$, so $|Y| < \ell$. According to property (iv) of the construction, the link graph $L_{G'}^Y(A_1 \cup C)$ consists of the $\binom{m}{\ell-|Y|}$ 
 smallest  $(\ell - |Y|)$-subsets under the antilexicographic order. Since the collection of all $(\ell - |Y|)$-subsets of $\{v_1, \dots, v_m\}$ constitutes an initial segment of this order, $e \setminus Y$ must be an edge in the link graph. This implies $e \in E(G')$, a contradiction. Hence, $A_4$ is a clique in $G'$.
\end{remark}

Next, we prove that 
every vertex in $V$ remains contained in at least $\binom{t}{k-1}$ copies of $K_k^\ell$ in $G'$.
First, we introduce some notation and necessary claims. For any $X \subseteq B$, we define:
\begin{align*}
\mathcal{K}_X &= \{ K \subseteq A_1 \cup A_2 : K \cup X \text{ induces a } K_k^\ell \text{ in } G \}, \\
\mathcal{K}'_X &= \{ K \subseteq A_1 \cup C : K \cup X \text{ induces a } K_k^\ell \text{ in } G' \}.
\end{align*}
For any $Y \subseteq B$ with $0 \le |Y| < \ell$, let:
\begin{align*}
\mathcal{S}_Y &= \{ S \subseteq A_1 \cup A_2 : L_G^Y(S) \text{ is a complete } (\ell - |Y|)\text{-graph} \}, \\
\mathcal{S}'_Y &= \{ S \subseteq A_1 \cup C : L_{G'}^Y(S) \text{ is a complete } (\ell - |Y|)\text{-graph} \}.
\end{align*}

The following claim  provides a useful characterization of the membership in  $\mathcal{K}_X$ and $\mathcal{K}'_X$ in terms of  the associated link graphs $\mathcal{S}_Y$ and  $\mathcal{S}_Y'$, respectively.

\begin{claim} \label{c}
For any $X \subseteq B$ and $T \subseteq A_1 \cup A_2$ with $|T| = k - |X|$, we have $T \in \mathcal{K}_X$ if and only if the following two conditions are satisfied: (i) If $|X| \ge \ell$, then $X$ is a clique in $G$;
(ii) $T \in \mathcal{S}_Y$ for every $Y \subseteq X$ with $\ell + |X| - k \le |Y| < \ell$.
\end{claim}

\begin{proof}
($\Leftarrow$) Suppose (i) and (ii) hold. To show $T \in \mathcal{K}_X$, we verify that an arbitrary $\ell$-subset $E \subseteq T \cup X$ is an edge in $G$. If $E \subseteq X$, then $|X| \ge \ell$, and $E \in E(G)$ follows immediately from  (i). Otherwise, let $Y_0 = E \cap X$. Since $E \not\subseteq X$, we have $0 \le |Y_0| < \ell$. Furthermore, as $E \setminus Y_0 \subseteq T$, the size of $Y_0$ satisfies:
\begin{equation*}
    |Y_0| = |E| - |E \setminus Y_0| \ge \ell - |T| = \ell - (k - |X|) = \ell + |X| - k.
\end{equation*}
By (ii), $T \in \mathcal{S}_{Y_0}$, which implies that the link graph $L_G^{Y_0}(T)$ is a complete $(\ell - |Y_0|)$-graph. Since $E \setminus Y_0$ is an $(\ell - |Y_0|)$-subset of $T$, it follows that $E \setminus Y_0 \in L_G^{Y_0}(T)$, and thus $E = Y_0 \cup (E \setminus Y_0)$ is an edge in $G$.

($\Rightarrow$) Conversely, assume $T \in \mathcal{K}_X$, meaning $T \cup X$ induces a $K_k^\ell$ in $G$. Since $X \subseteq T \cup X$, any $\ell$-subset of $X$ must be an edge in $G$, so (i) is satisfied. For any $Y \subseteq X$ with $|Y| < \ell$, consider an arbitrary $(\ell - |Y|)$-subset $S \subseteq T$. The set $S \cup Y$ is an $\ell$-subset of $T \cup X$, and thus $S \cup Y \in E(G)$ by the clique property of $T \cup X$. By the definition of the link graph, this implies $S \in E(L_G^Y(T))$. Since this holds for all such $S \subseteq T$, $L_G^Y(T)$ is a complete $(\ell - |Y|)$-graph, confirming that $T \in \mathcal{S}_Y$ for all valid $Y$, which satisfies (ii).
\end{proof}

An identical argument shows that Claim \ref{c} also holds for $\mathcal{K}_X'$, $\mathcal{S}_Y'$ and $G'$.

\begin{claim} \label{c'}
For any $X \subseteq B$ and $T \subseteq A_1 \cup C$ with $|T| = k - |X|$, we have $T \in \mathcal{K}_X'$ if and only if the following two conditions are satisfied: (i) If $|X| \ge \ell$, then $X$ is a clique in $G'$;
(ii) $T \in \mathcal{S}_Y'$ for every $Y \subseteq X$ with $\ell + |X| - k \le |Y| < \ell$.
\end{claim}

\begin{claim} \label{o}
For any $Y \subseteq B$ with $0 \le |Y| < \ell$, and any $S_1, S_2 \subseteq A_1 \cup C$ with $|S_1| = |S_2| = s$ and $S_1 < S_2$, if $S_2 \in \mathcal{S}'_Y$, then $S_1 \in \mathcal{S}'_Y$. Thus we get that the cliques of  size $s$ in $L_{G'}^Y(A_1\cup C)$ are the smallest ones with respect to the antilexicographic order. 
\end{claim}

\begin{proof}
If $Y = \emptyset$, the result is immediate from the construction of $G'$, as $A_1 \cup C$ induces a $K_{t+1}^\ell$ on its first $t+1$ vertices together with $t+1$ isolated vertices. Now suppose $Y \neq \emptyset$, and let $r = \ell - |Y|$, 
where $0 < r \leq \ell - 1$. Identifying vertices with their indices, let $S_1 = \{i_1, \dots, i_s\}$ and $S_2 = \{i'_1, \dots, i'_s\}$ be $s$-subsets such that $S_1 < S_2$. Since $S_2 \in \mathcal{S}'_Y$ and  by the definition of $\mathcal{S}'_Y$, we have $|S_2|\ge \ell-|Y|$, hence  $s\ge r$.  For any $r$-subset $W \subseteq S_1$, the definition of the antilexicographic order implies:
\[ W \le \{i_{s-r+1}, \dots, i_s\} \le \{i'_{s-r+1}, \dots, i'_s\}. \]
Since $S_2 \in \mathcal{S}'_Y$, the set $\{i'_{s-r+1}, \dots, i'_s\}$ is an edge in $L_{G'}^Y(A_1 \cup C)$. Property (iv) of the construction of $G'$ ensures that $W$  is an edge in $L_{G'}^Y(A_1 \cup C)$. By the arbitrariness of $W$, $S_1$ is a clique in $L_{G'}^Y(A_1 \cup C)$, hence $S_1 \in \mathcal{S}'_Y$. Thus we get that the cliques of size $s$ in $L_{G'}^Y(A_1\cup C)$ are the smallest ones with respect to the antilexicographic order.\end{proof}

 For any $u \in B$, let $k_u$ denote the number of copies of $K_k^\ell$ in $G$ that contain $u$ and have a nonempty intersection with $A_1 \cup A_2$, and let $k'_u$ denote the number of  copies of $K_k^\ell$ in $G'$ that contain $u$ and have a nonempty intersection with $A_1 \cup C$.

\begin{claim} \label{lem19}
For any $u \in B$, if $u$ is not contained in a $K_{t+1}^\ell$ in $G'$, then $k'_u \ge k_u$.
\end{claim}

\begin{proof}
Let $\mathcal{X}_u = \{X \subseteq B : u \in X, X \text{ is a clique in } G \text{ or } |X| < \ell\}$, and let $\mathcal{X}'_u$ be defined analogously for $G'$. From the construction of $G'$, it is straightforward to see that $\mathcal{X}_u = \mathcal{X}'_u$. Then $k_u$ is given by:
\begin{align*}
k_u &= \sum_{X \in \mathcal{X}_u} |\mathcal{K}_X| \\
&= \sum_{X \in \mathcal{X}_u} \left| \bigcap_{\substack{Y \subseteq X \\ \ell+|X|-k \le |Y| < \ell}} \{S \in \mathcal{S}_Y : |S| = k - |X|\} \right| \\
&\le \sum_{X \in \mathcal{X}_u} \min_{\substack{Y \subseteq X \\ \ell+|X|-k \le |Y| < \ell}} \left| \{S \in \mathcal{S}_Y : |S| = k - |X|\} \right|.
\end{align*}
The second equality follows from Claim \ref{c}. Similarly,  applying Claim \ref{c'}, we have:
\begin{align*}
k'_u &= \sum_{X \in \mathcal{X}_u} |\mathcal{K}'_X| \\
&= \sum_{X \in \mathcal{X}_u} \left| \bigcap_{\substack{Y \subseteq X \\ \ell+|X|-k \le |Y| < \ell}} \{S' \in \mathcal{S}'_Y : |S'| = k - |X|\} \right|.
\end{align*}

By Claim \ref{o}, for each $Y \subseteq X$ in the range  $\ell+|X|-k \le |Y| < \ell$, the cliques of  size $k-|X|$ in $L_{G'}^Y(A_1\cup C)$ form an initial segment with respect to the antilexicographic order. Since the intersection of such initial segments coincides with the one of minimum cardinality, it follows that:
\begin{align*}
k'_u 
&= \sum_{X \in \mathcal{X}_u} \min_{\substack{Y \subseteq X \\ \ell+|X|-k \le |Y| < \ell}} \left| \{S' \in \mathcal{S}'_Y : |S'| = k - |X|\} \right|.
\end{align*}

By (iv) of the construction of $G'$ and the Kruskal-Katona theorem (Theorem 2 (i)), for any non-empty $Y \subseteq B$, we have $|\{S' \in \mathcal{S}'_Y : |S'| = k-|X|\}| \ge |\{S \in \mathcal{S}_Y : |S| = k-|X|\}|$. Indeed, let $H_1 = L_G^Y(A_1 \cup A_2)$ and $H_2 = L_{G'}^Y(A_1 \cup C)$ be the respective link graphs.  By construction, we have $|E(H_1)| = |E(H_2)|$. For $i=1,2$, let $L_i$ denote the $(k-|X|)$-graph whose edge set $E(L_i)$ consists of the $(k-|X|)$-cliques in $H_i$. Noting that  $|\{S' \in \mathcal{S}'_Y : |S'|=k-|X|\}|=|E(L_2)|$ and $ |\{S \in \mathcal{S}_Y : |S| = k-|X|\}|=|E(L_1)|$, it suffices to show  that $|E(L_2)| \geq |E(L_1)|$. Suppose, for the sake of contradiction, that $|E(L_1)| > |E(L_2)|$. Let $L'_1$ be the antilexicographic compression of $L_1$ onto the vertex set $A_1 \cup C$. Thus, $|E(L'_1)| = |E(L_1)|$, and the edges of $L'_1$ form an initial segment of $(k-|X|)$-sets under the antilexicographic order. It follows from the Kruskal-Katona theorem (Theorem 2 (i)) that
\[ |E(H_1)| \ge |\partial_{\ell-|Y|} L_1| \ge |\partial_{\ell-|Y|} L'_1|. \]
 Recall that $E(H_2)$ is also an initial segment of $(\ell-|Y|)$-sets in the same order. Since  $|E(L'_1)| > |E(L_2)|$, the shadow $\partial_{\ell-|Y|} L'_1$ must strictly contain $E(H_2)$ as an initial segment. This implies $|\partial_{\ell-|Y|} L'_1| > |E(H_2)|$, leading to the chain of inequalities
\[ |E(H_1)| \ge |\partial_{\ell-|Y|} L'_1| > |E(H_2)|, \]
which contradicts the equality $|E(H_1)| = |E(H_2)|$. Therefore, $|E(L_2)| \ge |E(L_1)|$, as required.

For each $X \in \mathcal{X}_u$, let $Y_X \subseteq X$ be a subset that attains the minimum in the expression for $k'_u$, satisfying:
\begin{equation*}
    |\{S' \in \mathcal{S}'_{Y_X} : |S'| = k-|X|\}| = \min_{\substack{Y \subseteq X \\ \ell+|X|-k \le |Y| < \ell}} |\{S' \in \mathcal{S}'_Y : |S'| = k-|X|\}|.
\end{equation*}
Suppose first that $Y_X \neq \emptyset$ for all $X \in \mathcal{X}_u$. Based on the previously established bound $|\{S' \in \mathcal{S}'_{Y_X} : |S'| = k-|X|\}| \ge |\{S \in \mathcal{S}_{Y_X} : |S| =  k-|X|\}|$, we obtain the following chain of inequalities:

\begin{align*}
k'_u &= \sum_{X \in \mathcal{X}_u} |\{S' \in \mathcal{S}'_{Y_X} : |S'| = k - |X|\}| \\
&\ge \sum_{X \in \mathcal{X}_u} |\{S \in \mathcal{S}_{Y_X} : |S| = k - |X|\}| \\
&\ge \sum_{X \in \mathcal{X}_u} \min_{\substack{Y \subseteq X \\ \ell+|X|-k \le |Y| < \ell}} |\{S \in \mathcal{S}_Y : |S| = k - |X|\}| \ge k_u.
\end{align*}

In the remaining case, suppose $Y_X = \emptyset$ for some $X \in \mathcal{X}_u$. The condition $0 = |Y_X| \ge \ell + |X| - k$ forces $|X| \le k - \ell$. Note that the collection $\{S' \in \mathcal{S}'_{\emptyset} : |S'| = k - |X|\}$ consists of all $(k-|X|)$-subsets of $A_1$, denoted by $\binom{A_1}{k-|X|}$. By Claim \ref{o} and the minimality of $Y_X$, we have the inclusion $\binom{A_1}{k-|X|} \subseteq \{S' \in \mathcal{S}'_Y : |S'| = k - |X|\}$ for all $Y \subseteq X$. In particular, taking $Y = \{u\}$ yields $\binom{A_1}{k-|X|} \subseteq \{S' \in \mathcal{S}'_{\{u\}} : |S'| = k - |X|\}$. According to the definition of $\mathcal{S}'_{\{u\}}$, the set $\{u\} \cup A_1$ must induce a $K_{t+2}^\ell$ in $G'$, which implies that $u$ is contained in a $K_{t+1}^\ell$. This contradicts our initial assumption that $u$ is not contained in any $K_{t+1}^\ell$ in $G'$, completing the proof. 
\end{proof}

With the preceding claims established, we are now ready to verify that the vertex clique-degree condition is preserved in $G'$.

\begin{claim} \label{sh}
For any vertex $u \in V$, $u$ is contained in at least $\binom{t}{k-1}$ copies of $K_k^\ell$ in $G'$.
\end{claim}

\begin{proof}
Suppose first that $u \in A_1 \cup A_2$. By the construction of $G'$, both $A_1$ and $A_2$ induce a $K_{t+1}^\ell$ in $G'$. Since $u$ belongs to at least one such clique of size $t+1$, the number of $K_k^\ell$ containing $u$ is at least $\binom{(t+1)-1}{k-1} = \binom{t}{k-1}$.

Next, suppose $u \in V \setminus (A_1 \cup A_2 \cup B)$. Since $G'[V \setminus (A_1 \cup A_2)] = G[V \setminus (A_1 \cup A_2)]$ and no edge containing $u$ intersects $A_1 \cup A_2$ in $G$, the vertex $u$ is contained in at least $\binom{t}{k-1}$ copies of $K_k^\ell$ in $G'$.

Finally, we consider the case where $u \in B$. If $u$ is contained in a clique $K_{t+1}^\ell$ in $G'$, the result follows immediately as in the first case. Assume $u$ is not contained in any $K_{t+1}^\ell$ in $G'$. Recall that  $k_u$ and $k'_u$ denote the number of  copies of $K_k^\ell$ containing $u$ in $G$ and $G'$, respectively, that are incident with $A_1 \cup A_2$ (or $A_1 \cup C$). By Claim \ref{lem19}, we have $k'_u \ge k_u$. Since the cliques containing $u$ that do not intersect $A_1 \cup A_2$ are preserved, the inequality $k'_u \ge k_u$ ensures that $u$ continues to satisfy the clique-degree condition in $G'$.  
\end{proof}

Next, we construct another $\ell$-graph $G''$ from $G'$ by eliminating the auxiliary vertex set $C$. 

\begin{definition} \label{def22}
Let $H_0$ be a hypergraph on the vertex set $[n]$. For any $1 \le i < j \le n$ and each $e \in E(H_0)$, define the shifting operator $S_{i,j}$ as:
\[ 
S_{i,j}(e) = 
\begin{cases} 
e \cup \{i\} \setminus \{j\}, & \text{if } j \in e, \, i \notin e, \text{ and } e \cup \{i\} \setminus \{j\} \notin E(H_0), \\
e, & \text{otherwise.}
\end{cases}
\]
Let $S_{i,j}(H_0)$ denote the hypergraph with vertex set $V(H_0)$ and edge set $\{S_{i,j}(e) : e \in E(H_0)\}$. 
\end{definition}

\noindent \textit{Note.} It is clear that $V(S_{i,j}(H_0)) = V(H_0)$ and $|E(S_{i,j}(H_0))| = |E(H_0)|$ for any $1 \le i < j \le n$.

\begin{lemma} \label{f}
Let $H_0$ be an $\ell$-uniform hypergraph and $k \ge \ell$ be an integer. For $1 \le i < j \le n$ and any $u \in V(H_0)\setminus\{j\}$, the number of copies of $K_k^\ell$ containing $u$ in $S_{i,j}(H_0)$ is at least the number of  copies of $K_k^\ell$ containing $u$ in $H_0$.
\end{lemma}

\begin{proof}
Let $\mathcal{C}_u(H_0)$ denote the collection of $k$-subsets of $V(H_0)$ that contain $u$ and induce a $K_k^\ell$ in $H_0$. For any $K \in \mathcal{C}_u(H_0)$, let $K^{(i,j)} = (K \setminus \{j\}) \cup \{i\}$. 

By the definition of the shifting operator $S_{i,j}$, if $K$ is a clique in $H_0$ but not in $S_{i,j}(H_0)$, it must be the case that $j \in K$, $i \notin K$, and $K^{(i,j)}$ is not a clique in $H_0$. However, in this case, the shifting operator $S_{i,j}$ ensures that $K^{(i,j)}$ is a clique in $S_{i,j}(H_0)$. Since $u \in K$ and $u \neq j$, it follows that $u \in V(K^{(i,j)})$.

Thus, the map $\phi: \mathcal{C}_u(H_0) \to \mathcal{C}_u(S_{i,j}(H_0))$, defined by $\phi(K) = K^{(i,j)}$ if $K$ is not a clique in $S_{i,j}(H_0)$ and $\phi(K) = K$ otherwise, is an injection. This implies that the number of $K_k^\ell$ containing $u$ in $S_{i,j}(H_0)$ is at least as large as the number of those in $H_0$.
\end{proof}

Claim \ref{sh} implies that every vertex in $V$ is contained in at least $\binom{t}{k-1}$ copies of $K_k^\ell$ in $G'$, and it follows from Lemma \ref{f} that this lower bound on the clique degree is preserved for $G_{n-(t+1)+a}$.

Suppose $|A_1 \cap A_2| = a$. We relabel the vertices in $V(G')$ such that $V(G') \setminus ((A_2 \setminus A_1) \cup C) = [n - (t+1) + a]$, where $A_2 \setminus A_1 = [n] \setminus [n - (t+1) + a]$ and $C = [n + t + 1] \setminus [n]$. Let $G_0 = G'$. For each $1 \le i \le n - (t+1) + a$, we define 
\[ G_i = S_{i, n+1}(S_{i, n+2}(\dots(S_{i, n+t+1}(G_{i-1}))\dots)). \]

From the construction of $G_{n-(t+1)+a}$, we immediately get the following statement:

\begin{claim} \label{cs}
Let $D \subseteq V\cup C$ induce a $K_k^\ell$ in $G_{n-(t+1)+a}$. If $D\cap C\neq \emptyset$, then for any $i\in V\setminus (D\cup(A_2\setminus A_1))$ and $j\in D\cap C$, we have $D \cup \{i\}\setminus \{j\}$ induce a $K_k^\ell$ in $G_{n-(t+1)+a}$.
\end{claim}

Applying the edge bound on $G$ given by Lemma \ref{d} together with Claim \ref{cs}, we obtain the following property of $G_{n-(t+1)+a}$.

\begin{claim} \label{in}
Assume that \(t\ge k-1\)
and $n>\frac14(t+1)^2\binom{t-1}{\ell-2}+2t+1.
$
For any vertex set $D \subseteq V\cup C$ that induces a $K_k^\ell$ in $G_{n-(t+1)+a}$, we have $D \cap C = \emptyset$.
\end{claim}

\begin{proof}
Suppose, for the sake of contradiction, that $D \cap C \neq \emptyset$. Since $N_{G'}(A_2 \setminus A_1) \subseteq A_1$ and $C \cap A_1=\emptyset$, we have $N_{G'}(C)\cap (A_2 \setminus A_1)=\emptyset$. Furthermore, we have $N_{G_{n-(t+1)+a}}(C)\cap (A_2 \setminus A_1)=\emptyset$. Since $D \cap C \neq \emptyset$ and $D$ induces a $K_k^\ell$ in $G_{n-(t+1)+a}$, we have $D \cap (A_2 \setminus A_1) = \emptyset$. Let $d_0 = |D \cap V|$. According to Claim \ref{cs}, every $k$-subset of the form $\{D \cap V\} \cup F$ induces a $K_k^\ell$  for any $F \in \binom{V \setminus (D \cup (A_2 \setminus A_1))}{k-d_0}$.

Without loss of generality, assume $D \cap V = [d_0]$. By specifically choosing $F = \{d_0+1, \dots, k-1, j\}$ for each $j$ in the range $k \le j \le n - (t+1) + a$, it follows that $[k-1] \cup \{j\}$ induces a $K_k^\ell$ in $G_{n-(t+1)+a}$. Since each such clique provides $\binom{k-2}{\ell-2}$ edges containing both $x$ and $j$ for any fixed $x \in [k-1]$, we obtain the following lower bound on the degree of vertices $x \in [k-1]$:
\begin{equation} \label{eq:deg_bound}
\deg_{G_{n-(t+1)+a}}(x) \ge \binom{k-2}{\ell-1}+\sum_{j=k}^{n-(t+1)+a} \binom{k-2}{\ell-2} \ge (n - t - k) \binom{k-2}{\ell-2}+\binom{k-2}{\ell-1}.
\end{equation}

To reach a contradiction, we evaluate the sum $ \sum_{v \in V} \left( \deg_G(v) - \binom{t}{\ell-1} \right)$. By the Handshaking Lemma and the fact that $|E(G)| \ge |E(G_{n-(t+1)+a})|$, we have
\begin{equation*}
\sum_{v \in V} \left( \deg_G(v) - \binom{t}{\ell-1} \right) \ge \ell |E(G_{n-(t+1)+a})| - n \binom{t}{\ell-1}.
\end{equation*}
Expressing the right-hand side in terms of vertex degrees, and noting that, for every \(v \in V\), since \(v\) is contained in at least \(\binom{t}{k-1}\) copies of \(K_k^\ell\), the Kruskal--Katona theorem implies
\[
\deg_{G_{n-(t+1)+a}}(v) \ge \binom{t}{\ell-1},
\]
we have:
\begin{align*}
\sum_{v \in V} \left( \deg_G(v) - \binom{t}{\ell-1} \right) &\ge \sum_{v \in V} \left( \deg_{G_{n-(t+1)+a}}(v) - \binom{t}{\ell-1} \right) + \sum_{v \in C} \deg_{G_{n-(t+1)+a}}(v) \\
&\ge \sum_{x=1}^{k-1} \left( \deg_{G_{n-(t+1)+a}}(x) - \binom{t}{\ell-1} \right).
\end{align*}
Applying the bound from \eqref{eq:deg_bound} and the fact that $n > \frac{1}{4}(t+1)^2 \binom{t-1}{\ell-2} + 2t+1$, it follows that
\begin{align*}
\sum_{v \in V} \left( \deg_G(v) - \binom{t}{\ell-1} \right) &\ge (k-1) \left[ (n-t-k) \binom{k-2}{\ell-2}+\binom{k-2}{\ell-1} - \binom{t}{\ell-1} \right]\\
&>(k-1) \left[ \left(\frac{1}{4}(t+1)^2 \binom{t-1}{\ell-2} + t-k+1\right) \binom{k-2}{\ell-2}+\binom{k-2}{\ell-1} - \binom{t}{\ell-1} \right].
\end{align*}

Set
\[
\Sigma=\sum_{v\in V}\left(\deg_G(v)-\binom{t}{\ell-1}\right),
\qquad
A=\frac14(t+1)^2\binom{t-1}{\ell-2}.
\] If \(\ell=2\), then the standing assumption gives \(t\ge k-1\ge2\). Hence
\[
\Sigma>(k-1)(A-1)\ge 2(A-1)> A,
\]
contradicting Lemma~\ref{d}. It remains to consider \(\ell\ge3\). Since \(k\ge \ell+1\), we have
\[
\binom{k-2}{\ell-2}\ge2
\quad\text{and}\quad
\binom{k-2}{\ell-1}\ge0.
\]
Moreover, \(t-k\ge-1\). Also, since \(\ell-1\ge2\),
\[
\binom{t}{\ell-1}
=
\frac{t}{\ell-1}\binom{t-1}{\ell-2}
\le
\frac{t}{2}\binom{t-1}{\ell-2}
<
\frac14(t+1)^2\binom{t-1}{\ell-2}
=A.
\]
Using \(t-k\ge-1\), \(\binom{k-2}{\ell-2}\ge2\), and
\(\binom{k-2}{\ell-1}\ge0\), we obtain
\[
(A+t-k)\binom{k-2}{\ell-2}
+\binom{k-2}{\ell-1}
-\binom{t}{\ell-1}
\ge
2(A-1)-\binom{t}{\ell-1}.
\]
Therefore,
\[
\begin{aligned}
\Sigma
&>
(k-1)\left[
(A+t-k+1)\binom{k-2}{\ell-2}
+\binom{k-2}{\ell-1}
-\binom{t}{\ell-1}
\right] \\
&\ge
(k-1)\left(2A-\binom{t}{\ell-1}\right) \\
&>A.
\end{aligned}
\]
again contradicting Lemma~\ref{d}. Thus \(D\cap C=\emptyset\).
\end{proof}

Finally, we introduce the definition of $\mathcal{G}(G, A_1, A_2)$. Let $\mathcal{G}(G, A_1, A_2)$ be the $\ell$-graph induced by $V$ in $G_{n-(t+1)+a}$.

\begin{lemma} \label{def20}
Assume that \(t\ge k-1\). Let \(G\) be an extremal \(\ell\)-graph
for Problem~\ref{13} with
$
|V(G)|>\frac14(t+1)^2\binom{t-1}{\ell-2}+2t+1.
$
Let \(A_1,A_2\in\binom{V(G)}{t+1}\) be two vertex sets that both induce a
\(K^\ell_{t+1}\) in \(G\). Then 
the  $\ell$-graph $\mathcal{G}(G, A_1, A_2)$ satisfies properties (1)--(5).
\end{lemma}

\begin{proof}
    Straightforward verification shows that properties (1)--(4) hold for $\mathcal{G}(G, A_1, A_2)$. Regarding property (5), every vertex in $V$ is contained in at least $\binom{t}{k-1}$ copies of $K_k^\ell$ in $G_{n-(t+1)+a}$. By Claim \ref{in}, all such cliques are contained in $V$, and thus this property is preserved in the induced subgraph $\mathcal{G}(G, A_1, A_2)$. Consequently, $\mathcal{G}(G, A_1, A_2)$ satisfies properties (1)--(5), as required.
\end{proof}

Since $G$ is an extremal hypergraph for Problem \ref{13}, it follows that $|E(G'')| = |E(G')|=|E(G)|$, where $G'' = \mathcal{G}(G, A_1, A_2)$. It follows that $G$ contains no edge $e \subseteq A_1 \cup A_2$ with nonempty intersection with both $A_1 \setminus A_2$ and $A_2 \setminus A_1$; otherwise, removing such an edge in the initial phase of the construction would yield $|E(G')| < |E(G)|$.

Thus, we immediately obtain the following lemma:

\begin{lemma}\label{be} Assume that \(t\ge k-1\).
    Let $G$ be an extremal $\ell$-graph for Problem \ref{13} with $|V(G)|>\frac14(t+1)^2\binom{t-1}{\ell-2}+2t+1.
$ Let $A_1, A_2\in \binom{V(G)}{t+1}$ be two distinct vertex sets that both induce a $K_{t+1}^\ell$ in $G$. Then $G$ contains no edge $e \subseteq A_1 \cup A_2$ with nonempty intersection with both $A_1 \setminus A_2$ and $A_2 \setminus A_1$.
\end{lemma}

 \section{Proof of Theorem \ref{1}}

 We first prove a structural lemma concerning how copies of
\(K^\ell_{t+1}\) can intersect in an extremal \(\ell\)-graph.

\begin{lemma}\label{lem23}
Assume that \(t\ge k-1\). Let \(G\) be an extremal
\(\ell\)-graph for Problem \ref{13} with
\[
|V(G)|>\frac14(t+1)^2\binom{t-1}{\ell-2}+3t+1.
\]
Let \(\mathcal A\) be the family of all vertex sets that induce a
\(K^\ell_{t+1}\) in \(G\).  Then one of the following statements must be true:

\begin{enumerate}
    \item[(1)] There is at most one unordered pair \(\{A_1,A_2\}\subseteq\mathcal A\)
such that \(A_1\ne A_2\) and \(A_1\cap A_2\ne\emptyset\).
    \item[(2)] We have \(\ell\ge3\), and there exists an extremal \(\ell\)-graph $G'$ for Problem \ref{13} with $V(G')=V(G)$ that contains an isolated copy of $K_{t+1}^\ell$.
\end{enumerate}
\end{lemma}

\begin{proof}
Suppose, for a contradiction, that neither (1) nor (2) holds. Since
(1) fails, there exist two distinct unordered pairs
\(\{A_1,A_2\}\) and \(\{A_3,A_4\}\) in \(\mathcal A\) such that
\(A_1\cap A_2\ne\emptyset\) and \(A_3\cap A_4\ne\emptyset\).

Let \(A_1=\{v_1,v_2,\ldots,v_{t+1}\}\), and set
\(G''=\mathcal G(G,A_1,A_2)\). By Lemma~\ref{def20}, \(G''\) satisfies the
clique-degree condition and \(|E(G'')|\le |E(G)|\). Since \(G\) is extremal,
we have \(|E(G'')|=|E(G)|\), and hence \(G''\) is also extremal for
Problem~\ref{13}. For $i \in \{3, 4\}$, we define the shifted sets $A'_i$ as follows:
\[ 
A'_3 = (A_3 \setminus (A_1 \cup A_2)) \cup \{v_1, v_2, \dots, v_{|A_3 \cap (A_1 \cup A_2)|}\} 
\]
and
\[ 
A'_4 = (A_4 \setminus (A_1 \cup A_2)) \cup \{v_1, v_2, \dots, v_{|A_4 \cap (A_1 \cup A_2)|}\}. 
\]

Let \(H\) be the \(\ell\)-graph induced by \(V(G'')\setminus (A_2\setminus A_1)\)
in \(G''\). We claim that \(H\) is an extremal \(\ell\)-graph for Problem~\ref{13}
on its vertex set. Suppose not. Then there exists an extremal \(\ell\)-graph
\(H_1\) for Problem~\ref{13} on \(V(H)\) such that \(|E(H_1)|<|E(H)|\). Since
\[
|V(H)|\ge |V(G)|-t>
\frac14(t+1)^2\binom{t-1}{\ell-2}+2t+1>
\frac14(t+1)^2\binom{t-1}{\ell-2},
\]
Lemma~\ref{d} implies that \(H_1\) has a vertex \(u\) with
\[
\deg_{H_1}(u)=\binom{t}{\ell-1}.
\]
By the Lovász version of the Kruskal--Katona theorem (Theorem~2(ii)),
\(u\) is contained in a copy of \(K^\ell_{t+1}\) in \(H_1\). Relabeling
\(H_1\), if necessary, we may assume that this copy has vertex set \(A_1\).
Let \(G_1\) be the \(\ell\)-graph obtained from \(H_1\) by adding the
vertices in \(A_2\setminus A_1\) and all \(\ell\)-edges of the clique on
\(A_2\) that meet \(A_2\setminus A_1\). Then \(G_1\) is admissible for
Problem~\ref{13}, has vertex set \(V(G)\), and satisfies
\[
|E(G_1)|-|E(G'')|=|E(H_1)|-|E(H)|<0.
\]
Since \(|E(G'')|=|E(G)|\), this contradicts the extremality of \(G\).
Thus \(H\) is extremal.

Note that $A'_3, A'_4 \subseteq V(H)$ both induce a $K^\ell_{t+1}$ in $H$. Since all neighbors of $A_2 \setminus A_1$ in $G''$ belong to the $K_{t+1}^\ell$ induced by $A_1$, every vertex in $H$ is contained in at least $\binom{t}{k-1}$ copies of $K_k^\ell$.

We shall show that $H$ contains two intersecting copies of $K_{t+1}^\ell$. Since the pairs $\{A_1, A_2\}$ and $\{A_3, A_4\}$ are distinct, we may assume, without loss of generality, that $A_3 \notin \{A_1, A_2\}$. Observe that $A_3 \setminus (A_1 \cup A_2)$ must be non-empty; otherwise, there is an edge $e\subseteq A_3$ in $G$ that has a nonempty intersection with both $A_1\setminus A_2$ and $A_2\setminus A_1$, a contradiction with Lemma \ref{be}. If $A_3 \cap (A_1 \cup A_2) \neq \emptyset$, then the definitions of $A'_3$ and $A_1$ immediately imply that $A'_3 \cap A_1 \neq \emptyset$ and $A'_3 \neq A_1$. In this case, $A'_3$ and $A_1$ are two such intersecting copies in $H$. Alternatively, if $A_3 \cap (A_1 \cup A_2) = \emptyset$, the fact that $A_3$ and $A_4$ intersect implies that $A'_3$ and $A'_4$ are two intersecting copies of $K_{t+1}^\ell$ in $H$.

Without loss of generality, let \(A'_3\) and \(A'_4\) be two such intersecting
copies of \(K^\ell_{t+1}\) in \(H\). Relabeling \(A'_3\) and \(A'_4\), if
necessary, we may assume that \((A'_4\setminus A'_3)\cap A_1=\emptyset\).
Indeed, both \(A'_3\cap A_1\) and \(A'_4\cap A_1\) are initial segments of
\(A_1\); hence we may choose \(A'_4\) so that
\(A'_4\cap A_1\subseteq A'_3\cap A_1\). Since \(H\) is extremal and
\[
|V(H)|\ge |V(G)|-t>
\frac14(t+1)^2\binom{t-1}{\ell-2}+2t+1,
\]
Lemma~\ref{def20} applies to \(H\) and the two cliques \(A'_3,A'_4\). Set
\[
H'=\mathcal G(H,A'_3,A'_4).
\]
Then every vertex in \(H'\) remains contained in at least
\(\binom{t}{k-1}\) copies of \(K^\ell_k\).

Finally, let $H''$ be the $\ell$-graph induced by $V(H') \setminus (A'_4 \setminus A'_3)$ in $H'$. It follows that every vertex in $H''$ is contained in at least $\binom{t}{k-1}$ copies of $K_k^\ell$.

  Set
$s_1=|A_2\setminus A_1|$, $s_2=|A'_4\setminus A'_3|.$ 
We derive the desired contradiction by distinguishing two cases according to
the value of \(s_1+s_2\).

\medskip
\noindent \textbf{Case 1: $s_1 + s_2 \ge t+1$.}  Let
\[
d=2(t+1)-s_1-s_2 .
\]
Since each of the pairs \(\{A_1,A_2\}\) and \(\{A'_3,A'_4\}\) consists of two
distinct intersecting \((t+1)\)-sets, we have
$1\le s_1,s_2\le t$, and hence $2\le d\le t+1$.

Partition $V$ into two parts
\[
P=V\setminus\big((A_2\setminus A_1)\cup(A'_4\setminus A'_3)\big),
\qquad
Q=(A_2\setminus A_1)\cup(A'_4\setminus A'_3).
\]
Choose a subset $W_0\subseteq A_1$ with $|W_0|=d$, and choose a subset
$W_1\subseteq Q$ with
\[
|W_1|=s_1+s_2-(t+1).
\]
Let
\[
W=W_0\cup W_1,\qquad W'=Q\setminus W_1.
\]
Then
\[
|W|=|W'|=t+1.
\]
Define an $\ell$-graph $G^*$ on $V$ by
\[
E(G^*)=E(H'')\cup \binom{W}{\ell}\cup \binom{W'}{\ell}.
\]
Since $W_0\subseteq A_1$ and $A_1$ is a clique in $H''$, this construction is well-defined. Moreover, every vertex of $G^*$ is contained in at least
$
\binom{t}{k-1}
$
copies of $K_k^\ell$: vertices in $P$ keep this property from $H''$, while every vertex in $Q$ lies in one of the two copies induced by $W$ and $W'$.

Clearly, $|V(G^*)|=|V(G)|$. Recalling that $|E(G)|=|E(G'')|$ and $|E(H')|\le |E(H)|$, we have
\[
\begin{aligned}
|E(G^*)|-|E(G)|
&=|E(G^*)|-|E(G'')|  \\
&\le \big(|E(G^*)|-|E(H'')|\big)
   -\big(|E(G'')|-|E(H)|\big)
   -\big(|E(H')|-|E(H'')|\big)\\
&=
\left(2\binom{t+1}{\ell}-\binom{d}{\ell}\right)
-\left(\binom{t+1}{\ell}-\binom{t+1-s_1}{\ell}\right)
-\left(\binom{t+1}{\ell}-\binom{t+1-s_2}{\ell}\right)\\
&=
\binom{t+1-s_1}{\ell}
+\binom{t+1-s_2}{\ell}
-\binom{d}{\ell}.
\end{aligned}
\]
Let
\[
x=t+1-s_1,\qquad y=t+1-s_2.
\]
Then $x,y>0$ and $d=x+y$. By Vandermonde's identity,
\[
\binom{d}{\ell}
=
\sum_{i=0}^{\ell}\binom{x}{i}\binom{y}{\ell-i}
\ge
\binom{x}{\ell}+\binom{y}{\ell}.
\]
Therefore
\[
|E(G^*)|-|E(G)|\le 0.
\]

If the inequality is strict, then $|E(G^*)|<|E(G)|$, contradicting the extremality of $G$. Hence equality must hold. When $\ell=2$, this is impossible, since
\[
\binom{x+y}{2}-\binom{x}{2}-\binom{y}{2}=xy>0.
\]
Thus equality can occur only when $\ell\ge3$. 
Since $G^*$ satisfies the clique-degree condition and $|E(G^*)|=|E(G)|$, it is an extremal hypergraph for Problem~\ref{13}. Moreover, by construction, $W'$ induces an isolated copy of $K_{t+1}^{\ell}$ in $G^*$. Therefore (2) holds, contradicting our assumption that (2) fails.

\noindent \textbf{Case 2: $s_1 + s_2 < t+1$.} Set
\[
P=V\setminus\big((A_2\setminus A_1)\cup(A'_4\setminus A'_3)\big),
\qquad
Q=(A_2\setminus A_1)\cup(A'_4\setminus A'_3).
\]
Choose a subset \(W_0\subseteq A_1\) with
\[
|W_0|=t+1-s_1-s_2,
\]
and set \(W=W_0\cup Q\). Then \(|W|=t+1\).

We define an \(\ell\)-graph \(G^*\) on \(V\) as follows. Let
\(G^*[P]=H''\), let \(W\) induce a copy of \(K^\ell_{t+1}\) in \(G^*\),
and add no other edge meeting \(Q\). This is well-defined, since
\(W_0\subseteq A_1\subseteq V(H'')\) and \(A_1\) induces a copy of
\(K^\ell_{t+1}\) in \(H''\). By construction, \(G^*\) satisfies the
clique-degree condition and \(V(G^*)=V(G)\).

We now compare the number of edges. Since \(|E(G'')|=|E(G)|\) and
\(|E(H')|\le |E(H)|\), we have
\[
\begin{aligned}
|E(G^*)|-|E(G)|
&= |E(G^*)|-|E(G'')|  \\
&\le
\bigl(|E(G^*)|-|E(H'')|\bigr)
-\bigl(|E(G'')|-|E(H)|\bigr)
-\bigl(|E(H')|-|E(H'')|\bigr)  \\
&=
\left(\binom{t+1}{\ell}-\binom{t+1-s_1-s_2}{\ell}\right)
-\left(\binom{t+1}{\ell}-\binom{t+1-s_1}{\ell}\right)  \\
&\qquad
-\left(\binom{t+1}{\ell}-\binom{t+1-s_2}{\ell}\right)  \\
&=
\binom{t+1-s_1}{\ell}
+\binom{t+1-s_2}{\ell}
-\binom{t+1-s_1-s_2}{\ell}
-\binom{t+1}{\ell}.
\end{aligned}
\]
It remains to show that the last expression is negative. Indeed, by repeated
applications of Pascal's identity,
\[
\begin{aligned}
&\binom{t+1}{\ell}
-\binom{t+1-s_1}{\ell}
-\binom{t+1-s_2}{\ell}
+\binom{t+1-s_1-s_2}{\ell}  \\
&\qquad =
\sum_{i=0}^{s_1-1}\sum_{j=0}^{s_2-1}
\binom{t-1-i-j}{\ell-2}>0.
\end{aligned}
\]
Here all terms are nonnegative, and the term with \(i=j=0\) is positive.
Therefore
\[
|E(G^*)|-|E(G)|<0,
\]
contradicting the extremality of \(G\).

In Case 1 we either contradicted the extremality of \(G\) or obtained
alternative (2), contrary to the assumption that neither (1) nor (2) holds.
Case 2 gives a direct contradiction to the extremality of \(G\). Hence the
lemma follows.
\end{proof}

\medskip
\noindent \textbf{Proof of Theorem \ref{1}.} Suppose to the contrary that $n > \frac{1}{4}(t+1)^2 \binom{t-1}{\ell-2} + 3t+1$ and that no extremal hypergraph for Problem \ref{13} contains an isolated copy of $K_{t+1}^\ell$.  Let \(G\) be an extremal \(\ell\)-graph for Problem~\ref{13} with no isolated
copy of \(K^\ell_{t+1}\). Set
\[
V_1=\left\{v\in V(G):\deg_G(v)=\binom{t}{\ell-1}\right\},
\qquad
V_2=V(G)\setminus V_1.
\] Lemma \ref{d} then yields
\[
|V_2| \leq  \sum_{v \in V_2} \left( \deg_G(v) - \binom{t}{\ell-1} \right) = \sum_{v \in V(G)} \left( \deg_G(v) - \binom{t}{\ell-1} \right) \leq  \frac{1}{4}(t+1)^2 \binom{t-1}{\ell-2}.
\]
Consequently, $|V_1|=n-|V_2|>3t+1$. For each $v\in V_1$, the Lov\'{a}sz version of the Kruskal-Katona theorem (Theorem 2 (ii)) implies that $v$ is contained in a unique copy of $K_{t+1}^{\ell}$ in $G$.  It follows that there exist at least three distinct copies of
\(K^\ell_{t+1}\) in \(G\).  By our contrary assumption, alternative~(2) of
Lemma~\ref{lem23} cannot occur. Hence alternative~(1) of Lemma~\ref{lem23} holds, 
so among all
copies of \(K^\ell_{t+1}\) in \(G\),  there is at most one intersecting pair.
Therefore two of these three copies are vertex-disjoint.
Let $A'_1$ and $A'_2$ be the vertex sets of two such vertex-disjoint copies.

Let $G_1=\mathcal G(G,A'_1,A'_2)$. Since $A'_1\cap A'_2=\emptyset$, Property~(3) of the $\mathcal G$-transformation implies that $A'_2$ induces an isolated copy of $K_{t+1}^{\ell}$ in $G_1$. By Properties~(2) and~(5), the \(\ell\)-graph $G_1$ also satisfies the condition in Problem~\ref{13} and has no more edges than $G$. Since $G$ is extremal, $G_1$ is extremal as well. This contradicts the assumption that no extremal hypergraph contains an isolated copy of $K_{t+1}^{\ell}$.\qed

\section{Conclusion}
 
In Theorem~\ref{1}, we proved that, for every integer \(t\ge k-1\), there exists
an extremal hypergraph for Problem~\ref{13} containing an isolated copy of
\(K^\ell_{t+1}\) whenever $
n>\frac14(t+1)^2\binom{t-1}{\ell-2}+3t+1.$  It is natural to ask whether this structural property holds for \emph{every} extremal hypergraph.
 
\begin{conjecture}\label{conj:all-extremal}
For every fixed pair \(k>\ell\ge2\), there exists a constant
\(C=C(k,\ell)\) such that, for every integer \(t\ge k-1\) and every
\(n>Ct\), every extremal hypergraph for
Problem~\ref{13} contains an isolated copy of \(K^\ell_{t+1}\).
\end{conjecture}
 
In addition, it would also be natural to determine the extremal hypergraphs for Problem~\ref{13} when $n$ is relatively small, for instance when \(t+1\le n\le C(k,\ell)t\). In this regime, the iterative reduction provided by Theorem~\ref{1} is no longer applicable, and new techniques may be required.
 
One may also naturally inquire whether Conjecture~\ref{conj:all-extremal} extends to non-integer $t$. The answer is negative. Specifically, suppose $k = 3$, $\ell = 2$, and $\lceil t \rceil$ is an even integer such that $\lceil t \rceil + 2$ divides $n$. Let $G$ be the disjoint union of several copies of $K_{\lceil t \rceil+2} \setminus M$, where $M$ is a collection of $\frac{\lceil t \rceil+2}{2}$ pairwise disjoint edges. In this construction, $G$ is a $\lceil t \rceil$-regular graph where the neighborhood of each vertex contains exactly $\binom{\lceil t \rceil}{2} - \frac{\lceil t \rceil}{2}$ edges (i.e., each vertex is contained in that many triangles). If $\binom{t}{2} \leq \binom{\lceil t \rceil}{2} - \frac{\lceil t \rceil}{2}$, then $G$ serves as an extremal graph, as any extremal graph for Problem~\ref{13} must have a minimum degree of at least $\lceil t \rceil$. Consequently, $G$ provides a counterexample.  The existence of this specific counterexample, together with other structures of a similar nature, leads us to the following problem.
 
\begin{problem}
 For real \(t\ge k-1\) and fixed \(k>\ell\ge2\), is there a constant
\(C=C(k,\ell)\) such that, whenever \(n\ge Ct\), every extremal hypergraph for
Problem~\ref{13} contains a connected component of size either \(\lceil t\rceil+1\)
or \(\lceil t\rceil+2\)?
\end{problem}

\section*{Acknowledgements}

The authors would like to express their sincere gratitude to Yi Zhao from Georgia State University for introducing this problem to us and for his insightful suggestions.

\end{document}